\documentclass[smallcondensed]{svjour3}     
\smartqed  
\usepackage{graphicx}
\usepackage{amsmath}
\usepackage{amssymb}
\usepackage{enumerate}
\usepackage{hyperref}
\usepackage{natbib}
\begin{document}

\title{The decompositions and positive semidefiniteness of fourth-order conjugate partial-symmetric tensors with applications\thanks{This work was funded by the National Natural Science Foundation of China (Grant No. 11671217, No. 12071234).}
}

\titlerunning{Fourth-order conjugate partial-symmetric tensors}        

\author{Pengfei Huang         \and
        Qingzhi Yang 
}


\institute{Corresponding author. Pengfei Huang \at
              School of Mathematical
Sciences and LPMC, Nankai University, Tianjin 300071, P.R. China. \\
              \email{huangpf@mail.nankai.edu.cn}. ORCID: 0000-0003-3097-1804.
           \and
           Qingzhi Yang \at
              School of Mathematical
Sciences and LPMC, Nankai University, Tianjin 300071, P.R. China. \\
              \email{qz-yang@nankai.edu.cn}
}

\date{Received: date / Accepted: date}

\maketitle

\begin{abstract}
  Conjugate partial-symmetric (CPS) tensor is a generalization of Hermitian matrices. For the CPS tensor decomposition some properties are presented. For real CPS tensors in particular, we note the subtle difference from the complex case of the decomposition. In addition to traditional decompositions in the form of the sum of rank-one tensors, we focus on the orthogonal matrix decomposition of CPS tensors, which inherits nice properties from decomposition of matrices. It then induces a procedure that reobtain the CPS decomposable property of CPS tensors. We also discuss the nonnegativity of the quartic real-valued symmetric conjugate form corresponding to fourth-order CPS tensors in real and complex cases, and establish its relationship to different positive semidefiniteness based on different decompositions. Finally, we give some examples to illustrate the applications of presented propositions.
\end{abstract}

\section{Introduction}
Tensor decompositions have significant applications in computer vision, data mining, statistical estimation and so on. There have been rich researches on different decompositions and ranks correspondingly, such as CANDECOMP/PARAFAC (CP) decomposition, Tucker decomposition, matrix decomposition and so on \citep{sidiropoulos2017tensor,kolda2009tensor,cichocki2015tensor,jiang2015tensor}. These different decompositions help with revealing structure characterization of tensors from different aspects.

For tensors with symmetric properties, it is natural to expect the decomposition contains the same symmetric information. For instance, it is well known that for a symmetric tensor, there always exists a symmetric CP decomposition \citep{comon2008symmetric}. \cite{ni2019hermitian,nie2020hermitian} introduced the Hermitian tensors, which is closely related to the quantum physics. They asserted that for the Hermitian tensor, it still has a Hermitian decomposition. \cite{jiang2016characterizing,fu2018decompositions} studied the conjugate partial-symmetric (CPS) tensors, characterizing real-valued complex polynomial functions. They found that the CPS tensor also enjoys the CPS decompositions. CPS tensors are a special class of Hermitian tensors, thus their decompositions inherit some properties from the Hermitian tensors while preserving differences. For real Hermitian tensors, \cite{nie2020hermitian} fully characterized the ones that have real Hermitian decomposition. For real CPS tensors that are not symmetric, it may deserve discussion about the real decomposition form for them.

As the simplest even-order tensor next to the matrix, the fourth-order tensor can be used to model many multi-dimensional data from the real practice, including MIMO radar waveform minimization \citep{aubry2013ambiguity}, blind source identification \citep{de2007fourth}, quadratic eigenvalue problems \citep{tisseur2001quadratic}, among many others. It is expected to study whether some nice properties for matrix decompositions still hold for fourth-order tensors or not. The orthogonal matrix decomposition of fourth-order tensors based on the square unfolding is of particular interest in this sense. See for instance \citep{jiang2018low,de2007fourth,basser2007spectral}. As is well-known, polynomials can be represented by tensors and many of the application scenarios mentioned above are polynomial optimization problems. There is a considerable amount of research attention on polynomial optimization, including the approximation algorithms like \citep{luo2010semidefinite,ling2010biquadratic} and the will-studied lasserre hierarchy method \citep{lasserre2001global}. See \citep{laurent2009sums} and the reference therein for more information. And the nonnegativity is naturally associated with the quartic polynomial functions, corresponding to the fourth-order tensor, which is an important aspect of polynomials and closely related to optimizing the polynomial function. \cite{jiang2017cones} conducted a systematic study on the cones of real nonnegative quartic forms in the term of real fourth-order symmetric tensors, and they gave the relationship between positive semidefinite quartic forms and the sum-of-squares (SOS) of polynomials. For complex case, \cite{jiang2016characterizing} then raised the questions about the nonnegativity of the real-valued conjugate complex polynomials and SOS property naturally. So it is worth establishing a similar relationship as the real case in the term of fourth-order CPS tensors.

In this paper, we focus on the fourth-order CPS tensor. We further study the properties of the CPS decomposition of CPS tensors based on some existing results in literatures and particularly we note the special property in real CPS tensors. The orthogonal matrix decomposition for fourth-order CPS tensor is presented and we explore its link to the CPS decomposition. Based on different decompositions, we then study the positive semidefiniteness for CPS tensors on complex and real fields, respectively, which connect to the nonnegativity of quartic conjugate complex forms.

This paper is organized as follows. Several notations and definitions are provided in section \ref{sec:pre}. Some properties of two different decompositions for fourth-order CPS tensors, the CPS decomposition and the orthogonal matrix decomposition are presented in section \ref{sec:cps} and section \ref{sec:matrix}, respectively. In section \ref{sec:psd}, we study different positive semidefiniteness of CPS tensors in real and complex cases and their relational structures. Some examples of CPS tensors about their decomposition and positive semidefiniteness are given in section \ref{sec:app}. It is concluded in section \ref{sec:conclude}, with a list of some open problems for future work.

\section{Preparations}\label{sec:pre}
All tensors considered in this paper are fourth-order. Let $\mathbb{F}$ be $\mathbb{R}$ or $\mathbb{C}$. For any vector $x\in\mathbb{C}^n$, we denote $x^H:=\bar{x}^T$, where $\bar{x}$ denotes the conjugate of $x$. Its real and imaginary parts are denoted by $Re(x)$ and $Im(x)$, respectively. $\mathcal{S}_{\mathbb{F}}^{n}$ denotes the set of all $n$ by $n$ symmetric matrices with entries in $\mathbb{F}$ field. $\mathcal{S}_+^n~(\mathcal{H}_+^n)$ denotes the set of all $n$ by $n$ real symmetric (complex Hermitian) positive semidefinite matrices. $\mathcal{A}\in \mathbb{F}^{n^4}$ is a fourth-order n dimensional tensor and $\mathcal{A}_{ijkl}$ is an element of it. $\mathcal{A}$ is called symmetric tensor if $\mathcal{A}_{ijkl}$ is invariant for all permutations of $(i,j,k,l)$. The set of symmetric fourth-order tensors in $\mathbb{F}^{n^4}$ is denoted by $\mathbb{F}^{n^4}_s$. "$\boxtimes$" denotes the Kronecker product for matrices.

\cite{ni2019hermitian} introduced the Hermitian tensor and Hermitian decomposition. Definitions \ref{def:hermitian} and \ref{def:hermitian_decomp} below are explicitly for the fourth-order tensor.
\begin{definition}\label{def:hermitian}
$\mathcal{A}\in \mathbb{C}^{n^4}$ is called a Hermitian tensor if
\[\mathcal{A}_{i_1i_2j_1j_2}=\bar{\mathcal{A}}_{j_1j_2i_1i_2},\]
for all labels $i_1,~i_2,~j_1,~j_2$ in the range. A Hermitian tensor $\mathcal{A}\in \mathbb{R}^{n^4}$ is called a real Hermitian tensor. The set of all Hermitian tensors in $\mathbb{F}^{n^4}$ is denoted as $\mathbb{F}^{n^4}_h$.
\end{definition}
\begin{definition}\label{def:hermitian_decomp}
For a Hermitian tensor $\mathcal{A}\in \mathbb{C}^{n^4}_h$, if it can be written as
\begin{equation*}\label{equ:hermitian_decomp}
\mathcal{A}=\sum\limits_{i=1}^r\lambda_iu_i^1\otimes u_i^2\otimes \bar{u}_i^1\otimes\bar{u}_i^2,
\end{equation*}
where $\lambda_i\in \mathbb{R}$ and $u_i^j\in\mathbb{C}^n$, for $i=1,2,\cdots,r$, then $\mathcal{A}$ is called Hermitian decomposable. The smallest $r$ is called the Hermitian rank of $\mathcal{A}$, which is denoted as $rank_{h}(\mathcal{A})$. The Hermitian decomposition with minimum $r$ is called a Hermitian rank decomposition of $\mathcal{A}$.
\end{definition}
\begin{definition}\cite[Definition 3.1]{nie2020hermitian}
A tensor $\mathcal{A}\in\mathbb{R}_h^{n^4}$ is called $\mathbb{R}$-Hermitian decomposable if
\[\mathcal{A}=\sum\limits_{i=1}^r\lambda_iu_i^1\otimes u_i^2\otimes u_i^1\otimes u_i^2,~\lambda_i\in \mathbb{R},~u_i^j\in\mathbb{R}^n,~i=1,2,\cdots,r.\]
The smallest such $r$ is called the $\mathbb{R}$-Hermitian rank of $\mathcal{A}$, which we denoted by $rank_{h\mathbb{R}}(\mathcal{A})$. The $\mathbb{R}$-Hermitian decomposition with minimum $r$ is called a $\mathbb{R}$-Hermitian rank decomposition of $\mathcal{A}$.
\end{definition}

A special class of Hermitian tensors, called conjugate symmetric tensors, is introduced by \cite{jiang2016characterizing} as follows.
\begin{definition}
$\mathcal{A}\in \mathbb{F}^{n^4}$ is called conjugate partial-symmetric (CPS) if it is Hermitian and
\begin{equation*}
\mathcal{A}_{ijkl}=\mathcal{A}_{\pi(ij)\pi(kl)},~1\le i,j,k,l\le n,
\end{equation*}
where $\pi(ij)$ and $\pi(kl)$ are any permutation of $\{i,j\}$ and $\{k,l\}$, respectively. The set of CPS tensors in $\mathbb{F}^{n_4}$ is denoted by $\mathbb{F}^{n^4}_{cps}$.
\end{definition}

For two tensors $\mathcal{A},~\mathcal{B}\in \mathbb{C}^{n^4}_{cps}$, their inner product is defined as
\[\left<\mathcal{A},\mathcal{B}\right>=\sum\limits_{i,j,k,l}\mathcal{A}_{ijkl}\bar{\mathcal{B}}_{ijkl}.\]
The Frobenius norm of $\mathcal{A}$ is accordingly defined as $\|\mathcal{A}\|_F:=\sqrt{\left<\mathcal{A},\mathcal{A}\right>}$. For complex matrix $X\in\mathbb{C}^{n\times n}$, $\left<X,\mathcal{A}X\right>=\left<\mathcal{A},\bar{X}\otimes X\right>=\sum_{ijkl}\mathcal{A}X_{ij}\bar{X}_{kl}$. "$\otimes$" denotes the tensor product. The quartic real-valued symmetric conjugate form \citep{jiang2016characterizing} corresponding to $\mathcal{A}\in\mathbb{C}_{cps}^{n^4}$ is defined as
\[\mathcal{A}(x,x,\bar{x},\bar{x})=\left<\mathcal{A},\bar{x}\otimes\bar{x}\otimes x\otimes x\right>
=\sum\limits_{i,j,k,l=1}^n\mathcal{A}_{ijkl}x_ix_j\bar{x}_k\bar{x}_l,~x\in\mathbb{C}^n.\]
For brevity, we use $x^{\otimes d}$ to stand for $\underbrace{x\otimes\cdots\otimes x}_{d}$.

The CP rank \citep{kolda2009tensor} of $\mathcal{A}\in \mathbb{C}^{n^4}$ is the smallest $r$ such that
\[\mathcal{A}=\sum\limits_{i=1}^ru_i^1\otimes u_i^2\otimes u_i^3\otimes u_i^4,\]
where $u^j_i\in\mathbb{C}^n$ for all $i,j$ in the range. We denote it as $rank(\mathcal{A})$.

\section{CPS decomposition and cps rank}\label{sec:cps}
This section studies some properties of the CPS decomposition for the CPS tensor. First, let us recall the following decomposition theorems for Hermitian and CPS tensors in fourth-order case.

\begin{theorem}\citep[Theorem 5.1]{ni2019hermitian}
Every Hermitian tensor $\mathcal{A}\in\mathbb{C}^{n^4}_h$ is Hermitian decomposable.
\end{theorem}
\begin{theorem}\citep[Theorem 3.2]{fu2018decompositions}\label{thm:cps_decomp}
A fourth-order tensor $\mathcal{A}\in \mathbb{C}^{n^4}$ is CPS if and only if it has the following CPS decomposition
\begin{equation*}
\mathcal{A}=\sum\limits_{i=1}^r\lambda_i a_i\otimes a_i\otimes \bar{a}_i\otimes\bar{a}_i,
\end{equation*}
where $\lambda_i\in \mathbb{R}$ and $a_i\in\mathbb{C}^n$ for $i=1,2,\cdots,r$.
\end{theorem}
The smallest $r$ is called the CPS rank of $\mathcal{A}$ \citep{fu2018decompositions}, denoted as $rank_{cps}(\mathcal{A})$. Similar to the Hermitian rank decomposition, we call the CPS decomposition with minimum $r$ a CPS rank decomposition of $\mathcal{A}$.
In section \ref{sec:matrix} we will give a procedure to obtain a CPS decomposition based on the orthogonal matrix decomposition.
\subsection{Basic properties}
It is obvious that $rank(\mathcal{A})\le rank_h(\mathcal{A})\le rank_{cps}(\mathcal{A})$ for any $\mathcal{A}\in \mathbb{C}_{cps}^{n^4}$. It is hard in general to determine the Hermitian or CPS rank. \cite{nie2020hermitian} explored the Hermitian rank decompositions for basis tensors of the Hermitian tensor. We discuss the CPS rank decompositions for basis tensors in the CPS tensor analogously.

For a scalar $c\in \mathbb{C}$, denote $\mathcal{E}^{ijkl}(c)$ the CPS tensor in $\mathbb{C}^{n^4}$ such that all entries are zeros, except $(\mathcal{E}^{ijkl}(c))_{ijkl}=(\mathcal{E}^{ijkl}(c))_{jikl}
=(\mathcal{E}^{ijkl}(c))_{ijlk}=\overline{(\mathcal{E}^{ijkl}(c))_{klij}}=c$, $i,j,k,l=1,2,\cdots,n$.
$\mathbb{C}^{n^4}_{cps}$ is a vector space over $\mathbb{R}$, adopting the standard scalar multiplication and addition. The set
\begin{equation*}
E:=\{\mathcal{E}^{ijij}(1)\}_{1\le i\le j\le n}\cup\{\mathcal{E}^{ijkl}(1),\mathcal{E}^{ijkl}(\sqrt{-1})\}_{1\le i\le j\le n,i\le k\le l\le n,\{i,j\}\neq\{k,l\}}.
\end{equation*}
is the canonical basis for $\mathbb{C}^{n^4}_{cps}$. Its dimension is
\[\dim \mathbb{C}^{n^4}_{cps}=\frac{n(n+1)}{2}+\frac{n(n+1)}{2}(\frac{n(n+1)}{2}-1)=(\frac{n(n+1)}{2})^2.\]

\begin{proposition}\label{pro:basis_decomp}
$rank_{cps}(\mathcal{E}^{iiii}(1))=1$; $rank_{cps}(\mathcal{E}^{iijj}(1))=4~(i\neq j)$ and it has the following CPS rank decomposition,
\begin{equation}\label{equ:basis_decomp}
\begin{aligned}
\mathcal{E}^{iijj}(1)=&\frac{1}{4}[(e_i+e_j)^{\otimes 4}+(e_i-e_j)^{\otimes 4}\\
&-(e_i+ie_j)^{\otimes 2}\otimes(e_i-ie_j)^{\otimes 2}-(e_i-ie_j)^{\otimes 2}\otimes(e_i+ie_j)^{\otimes 2}],
\end{aligned}
\end{equation}
where $e_i$ is the ith identity vector.
\end{proposition}
\begin{proof}
$rank_{cps}(\mathcal{E}^{iiii}(1))=1$ is obvious, since $\mathcal{E}^{iiii}(1)=e_i^{\otimes 4}$. According to \cite[Theorem 2.3]{nie2020hermitian}, \eqref{equ:basis_decomp} is also a Hermitian rank decomposition for $\mathcal{E}^{iijj}(1)$, and $rank_h(\mathcal{E}^{iijj}(1))=4$ when $i\neq j$. Thus $rank_h(\mathcal{E}^{iijj}(1))\le rank_{cps}(\mathcal{E}^{iijj}(1))=4$.
\end{proof}

From Proposition \ref{pro:basis_decomp}, we have $rank_{cps}(\mathcal{E}^{iijj}(1))>rank(\mathcal{E}^{iijj}(1))=2~(i\neq j)$, since $\mathcal{E}^{iijj}(1)$ has a CP rank decomposition $\mathcal{E}^{iijj}(1)=e_i^{\otimes 2}\otimes e_j^{\otimes 2}+e_j^{\otimes 2}\otimes e_i^{\otimes 2}$.
For $\mathcal{E}^{iiij}(1)$, $i\neq j$, we can find a Hermitian decomposition of it,
\begin{equation*}
\begin{aligned}
\mathcal{E}^{iiij}(1)=&\frac{1}{2}[(e_i+e_j)\otimes e_i\otimes(e_i+e_j)\otimes e_i-(e_i-e_j)\otimes e_i\otimes(e_i-e_j)\otimes e_i\\
&+e_i\otimes(e_i+e_j)\otimes e_i\otimes(e_i+e_j)-e_i\otimes(e_i-e_j)\otimes e_i\otimes(e_i-e_j)].
\end{aligned}
\end{equation*}
It has been proved that $rank(\mathcal{E}^{iiij}(1))=rank_h(\mathcal{E}^{iiij}(1))=rank_{h\mathbb{R}}(\mathcal{E}^{iiij}(1))=4$ \citep{nie2020hermitian}. However we cannot figure out what the CPS rank of $\mathcal{E}^{iiij}(1)$.

In general, we can only determine the CPS rank in some special cases. There has been various conditions for the uniqueness and characterization of the CP rank decomposition, see the survey \citep{sidiropoulos2017tensor} for instance. The Kruskal-type theorem is one of the most well-known result. Let $k_U$ be the Kruskal rank of set $U$, which is the maximum number $k$ such that every subset of $k$ vectors in $U$ is linearly independent. For a CPS tensor, given a CPS decomposition, the following proposition gives a condition to judge whether that is a CPS rank decomposition.
\begin{proposition}
Let $\mathcal{A}=\sum_{i=1}^r\lambda_iu_i^{\otimes 2}\otimes\bar{u}_i^{\otimes2}$ be a CPS tensor, with $0\neq \lambda_i\in\mathbb{R}$. Let $U=\{u_1,\cdots,u_r\}$. If $2k_{U}\ge r+2$, then $rank_{cps}(\mathcal{A})=r$, and the CPS rank decomposition of $\mathcal{A}$ is unique up to permutation and scaling of decomposing vectors.
\end{proposition}
\begin{proof}
If $2K_U\ge r+2$, then
\[K_U+K_U+K_{\bar{U}}+K_{\bar{U}}>2r+3,\]
which satisfies the condition of the Kruskal theorem \citep{sidiropoulos2000uniqueness} for the uniqueness of the CP rank decomposition. The proposition holds naturally.
\end{proof}
\subsection{Real CPS decomposition}
For real CPS tensors, one may expect that the decomposition only involves real factors. However, $\sum_i\lambda_ia_i\otimes a_i\otimes\bar{a}_i\otimes\bar{a}_i$ can only represent real symmetric tensors when $\lambda_i\in R$, $a_i\in\mathbb{R}^n$. So the decomposition for real CPS tensors might has subtle differences from the complex case. As an example, we have the following proposition.
\begin{proposition}
For a real CPS tensor $\mathcal{A}\in\mathbb{R}_{cps}^{n^4}\backslash\mathbb{R}_s^{n^4}$, it holds $rank_{cps}(\mathcal{A})\ge 2$. That is, $\mathcal{A}$ cannot be decomposed as $\mathcal{A}=\lambda a\otimes a\otimes\bar{a}\otimes\bar{a}$ with $\lambda\in\mathbb{R},~a\in\mathbb{C}^n$.
\end{proposition}
\begin{proof}
Suppose that $\mathcal{A}=\lambda a\otimes a\otimes\bar{a}\otimes\bar{a}$, where $\lambda\in\mathbb{R}$, $a\in\mathbb{C}^n$. Let $a_i=|a_i|e^{i\theta_i}$, then we have
\begin{equation*}
\mathcal{A}_{iiij}=|a_i|^3|a_j|e^{\theta_i-\theta_j}=|a_i|^3|a_j|e^{\theta_j-\theta_i}=\mathcal{A}_{ijii}.
\end{equation*}
Thus $\theta_i-\theta_j=k\pi,~k\in\mathbb{Z}$ when $|a_i||a_j|\neq0$. With not loss of generality, we can assume $a=e^{i\alpha}b$, where $b$ is given as $b_i=|a_i|e^{ik_i\pi},~k_i\in\mathbb{Z}$. Then $b$ is a real vector and $\mathcal{A}=\lambda b^{\otimes 4}$ is a real symmetric tensor. It contradicts our assumption in the proposition. Thus $rank_{cps}(\mathcal{A})\ge 2$.
\end{proof}
We can also obtain that $rank(\mathcal{A})\ge 2$ when $\mathcal{A}\in\mathbb{R}_{cps}^{n^4}\backslash\mathbb{R}_s^{n^4}$ analogously.

Nie et al. have shown that not every real Hermitian tensor is $\mathbb{R}$-Hermitian decomposable \citep{nie2020hermitian}. We recall the characterization of a $\mathbb{R}$-Hermitian decomposable tensors in $\mathbb{R}^{n^4}$ given by them as the following theorem.
\begin{theorem}\citep{nie2020hermitian}\label{thm:real_decomp}
A tensor $\mathcal{A}\in\mathbb{R}_h^{n^4}$ is $\mathbb{R}$-Hermitian decomposable,
if and only if
\begin{equation}\label{equ:real_hdecomp}
\mathcal{A}_{i_1i_2j_1j_2}=\mathcal{A}_{k_1k_2l_1l_2}
\end{equation}
for all labels such that $\{i_s,j_s\}=\{k_s,l_s\},~s=1,2$.
\end{theorem}

Based on it, we obtain the really partially symmetric decomposition for real CPS tensors.
\begin{theorem}\label{thm:real_cpsdecomp}
A real CPS tensor $\mathcal{A}\in \mathbb{R}^{n^4}_{cps}$ has the following really partially symmetric decomposition
\[\mathcal{A}=\sum\limits_{i=1}^r\lambda_ia_i^{\otimes 2} \otimes b_i^{\otimes 2},\]
where $\lambda_i\in\mathbb{R}$ and $a_i,~b_i\in\mathbb{R}^n$, $i=1,2,\cdots,r$.
\end{theorem}
\begin{proof}
Let $\mathcal{\hat{A}}$ be a real fourth-order tensor, given by $\mathcal{\hat{A}}_{ijkl}=\mathcal{A}_{ikjl},~1\le i,j,k,l\le n$. Then we have
\[\mathcal{\hat{A}}_{klij}=\mathcal{A}_{kilj}=\mathcal{A}_{ikjl}=\mathcal{\hat{A}}_{ijkl}.\]
So $\mathcal{\hat{A}}$ is a real Hermitian tensor. It also holds that
\[\mathcal{\hat{A}}_{kjil}=\mathcal{A}_{kijl}=\mathcal{A}_{ikjl}=\mathcal{\hat{A}}_{ijkl}.\]
Similarly, we have $\mathcal{\hat{A}}_{ilkj}=\mathcal{\hat{A}}_{ijkl}$, then $\mathcal{\hat{A}}$ satisfies the condition \eqref{equ:real_hdecomp}. According to Theorem \ref{thm:real_decomp}, there exists the decomposition $\mathcal{\hat{A}}=\sum_{i=1}^r\lambda_i a_i\otimes b_i\otimes a_i\otimes b_i$. Thus we obtain a really partially symmetric decomposition for $\mathcal{A}$,
\[\mathcal{A}=\sum_{i=1}^r\lambda_i a_i^{\otimes 2}\otimes b_i^{\otimes 2},\]
where $\lambda_i\in\mathbb{R}$, $a_i,~b_i\in\mathbb{R}^n$, for $i=1,2,\cdots,r$.
\end{proof}

In the next section, we study the orthogonal matrix decomposition of the CPS tensor. It then induce a decomposition for the real CPS tensor in the form of the sum of low-rank real CPS tensors.

\section{Orthogonal matrix decompositions}\label{sec:matrix}
For fourth-order tensor, the square unfolding form $M(\mathcal{A})$ has particular research interest. It is defined as
\begin{equation*}
M(\mathcal{A})_{(j-1)n+i,(l-1)n+k}=\mathcal{A}_{ijkl}.
\end{equation*}
\cite{jiang2018low,jiang2015tensor} and \cite{de2007fourth} introduced the matrix decomposition (or M-decomposition) based on it.  \cite{nie2020hermitian} emphasized the orthogonality of such kind of decomposition for Hermitian tensors. We summarize and rewrite them for fourth-order CPS tensors in the following theorem and give a more concise proof for it.
\begin{theorem}[M-decomposition]\label{thm:matrix_decomp}
If $\mathcal{A}\in \mathbb{C}^{n^4}$ is a CPS tensor, then it has an orthogonal matrix decomposition
\begin{equation}\label{equ:matrix_decomp}
\mathcal{A}=\sum\limits_{i=1}^r\lambda_iE_i\otimes \bar{E_i},
\end{equation}
where $\lambda_i\in \mathbb{R}$, $E_i\in \mathcal{S}_{\mathbb{C}}^{n}$ are complex symmetric matrices. $\left<E_i,E_j\right>=0$ for $i\neq j$ and $\left<E_i,E_i\right>=1$. The decomposition is unique when $\lambda_i$ are distinct from each other. If $\mathcal{A}$ is a real CPS tensor, then $E_i$ of \eqref{equ:matrix_decomp} are real symmetric matrices.
\end{theorem}
\begin{proof}
Since $\mathcal{A}$ is CPS, $M(\mathcal{A})$ is Hermitian, and can be decomposed as
\[
M(\mathcal{A}) = \sum_{k = 1}^r \lambda_{k} e_{k} e_{k}^*,\]
where $\lambda_{k}\in \mathbb{R}$ and $e_{k}\in\mathbb{C}^{n^2}$ are mutually orthogonal. Folding $e_{k}$ into matrix $E_{k}$ via $(E_{k})_{ij} = (e_{k})_{(j-1)\times n + i}$, then $E_{k}$ are  mutually orthogonal, that is $\langle E_k, E_l \rangle = 0$ for $k\neq l$ and $\langle E_k, E_k \rangle = 1$, for $k,~l=1,2,\cdots, r$. Thus, we have $\mathcal{A} = \sum\limits_{i = k}^r \lambda_k E_k\otimes \bar{E}_k$.

We also have $M(\mathcal{A})e_\tau = \lambda_\tau e_\tau$, for $1\le\tau\le r$, that is, $\sum_{k,l=1}^na_{ijkl}(e_\tau)_{(l-1)\times n + k} = \lambda_\tau (e_\tau)_{(j-1)\times n + i}$, for any $1\le i, j \le n$. Since $a_{ijkl} = a_{jikl}$, then $(e_\tau)_{(j-1)\times n + i} =(e_\tau)_{(i-1)\times n + j} $, thus $E_\tau$ is symmetric. The uniqueness of the decomposition follows the property of spectral decomposition of Hermitian matrix naturally.
\end{proof}
\subsection{Link between M-decomposition and CPS decomposition}
In this subsection, based on the orthogonal matrix decomposition, we reobtain the CPS decomposability of CPS tensors by a construct procedure in a straightforward and easy-to-understand way. Firstly, we induce a low-rank decomposition for CPS tensors.
\begin{lemma}[SSVD of complex matrix]\citep[Corollary 2.6.6,Corollary 4.4.4]{horn2012matrix}\label{lem:complexsym}
Let $A\in\mathcal{S}_{\mathbb{C}}^n$, it has a symmetric singular value decomposition (SSVD) $A=U\Sigma U^T$, where $U$ is a unitary matrix and $\Sigma$ is a nonnegative diagonal matrix whose diagonal entries are the singular values of $A$.
\end{lemma}
\begin{corollary}\label{thm:lowcpsdecomp}
Let $\mathcal{A}\in \mathbb{C}_{cps}^{n^4}$, it can be decomposed as the sum of simple low rank CPS tensors,
\begin{equation}\label{equ:lowcpsdecomp}
\mathcal{A}=\sum\limits_{i=1}^R\alpha_i(p_i\otimes p_i\otimes \bar{q}_i\otimes \bar{q}_i+q_i\otimes q_i\otimes \bar{p}_i\otimes \bar{p}_i),
\end{equation}
where $\alpha_i\in\mathbb{R}$, $p_i,~q_i\in\mathbb{C}^n$, for $i=1,2,\cdots,R$. If $\mathcal{A}\in \mathbb{R}_{cps}^{n^4}$, then $p_i,~q_i\in\mathbb{R}^n$ for $i=1,2,\cdots,R$, and $\mathcal{A}$ is represented as the sum of low-rank real CPS tensors. We call \eqref{equ:lowcpsdecomp} the extended "rank-one" decomposition of the CPS tensor.
\end{corollary}
\begin{proof}
From Theorem \ref{thm:matrix_decomp}, $\mathcal{A} = \sum\limits_{i = 1}^r \lambda_i E_i\otimes \bar{E}_i$, where $E_i$ are symmetric. According to Lemma \ref{lem:complexsym}, each $E_i$ can be decomposed as $\sum_{j = 1}^{r_i}\beta_i^j u_i^j(u_i^j)^\top$, where $\beta_i^j\in\mathbb{R},~u_i^j\in\mathbb{C}^n$. Thus
\begin{equation}\label{equ:proof}
\begin{aligned}
\mathcal{A} &= &&\sum\limits_{i = 1}^r \lambda_i (\sum_{j = 1}^{r_i} \beta_i^j u_i^j(u_i^j)^\top)\otimes \overline{(\sum_{k = 1}^{r_i} \beta_i^k u_i^k(u_i^k)^\top)}  \\
                  & = &&\sum_{i = 1}^r \lambda_i (\sum_{j = 1}^{r_i}\sum_{k = j}^{r_i} \beta_i^j \beta_i^k(u_i^j \otimes u_i^j \otimes \bar{u}_i^k \otimes \bar{u}_i^k + u_i^k \otimes u_i^k \otimes \bar{u}_i^j \otimes \bar{u}_i^j).
\end{aligned}
\end{equation}
The desired decomposition form then follows. If $\mathcal{A}\in\mathbb{R}_{cps}^{n^4}$, according to Theorem \ref{thm:matrix_decomp}, $E_i$ are real symmetric matrices, which have spectral decomposition as $\sum_{j = 1}^{r_i}\beta_i^j u_i^j(u_i^j)^\top$, where $\beta_i^j\in\mathbb{R},~u_i^j\in\mathbb{R}^n$. The rest proof for the real case is the same.
\end{proof}

Let $\mathcal{A}\in\mathbb{C}_{cps}^{n^4}$, if $\mathcal{A}=\sum_{i=1}^ra_i^{\otimes 2}\otimes\bar{a}_i^{\otimes 2}$ with $r\le n(n+1)/2$ and $\{\Phi(a_ia_i^T,a_ja_j^T)\}_{1\le i<j\le r}$ are linearly independent, where $\Phi(X,Y)=X_{ik}Y_{jl}+Y_{ik}X_{jl}-X_{il}Y_{jk}-Y_{il}X_{jk}$, we can compute the CPS decomposition based on the above orthogonal matrix decomposition as \citep{de2007fourth,de2006link} using simultaneous matrix diagonalization. \cite{fu2018decompositions} developed a constructive algorithm based on \eqref{equ:matrix_decomp} using the representation of random vectors, which is complicated. According to Corollary \ref{thm:lowcpsdecomp}, we have a basic way to obtain an explicit CPS decomposition from the extended "rank-one" CPS decomposition as follows.
\begin{corollary}\label{thm:construct}
Let $\mathcal{A}\in\mathbb{C}_{cps}^{n^4}$, $\mathcal{A}=\sum_{i=1}^k\alpha_i(p_i\otimes p_i\otimes \bar{q}_i\otimes \bar{q}_i+q_i\otimes q_i\otimes \bar{p}_i\otimes \bar{p}_i)$ for some $k$, then it has a CPS decomposition as
\begin{equation*}
\begin{aligned}
\mathcal{A}=&\sum\limits_{i=1}^k\frac{\alpha_i}{4}[(p_i+q_i)^{\otimes 2}\otimes(\bar{p}_i+\bar{q}_i)^{\otimes 2}+(p_i-q_i)^{\otimes 2}\otimes(\bar{p}_i-\bar{q}_i)^{\otimes 2}\\
&-(p_i+iq_i)^{\otimes 2}\otimes(\bar{p}_i-i\bar{q}_i)^{\otimes 2}-(p_i-iq_i)^{\otimes 2}\otimes(\bar{p}_i+i\bar{q}_i)^{\otimes 2}].
\end{aligned}
\end{equation*}
\end{corollary}
\begin{proof}
For a CPS tensor $\mathcal{E}=p\otimes p\otimes \bar{q}\otimes\bar{q}+q\otimes q\otimes \bar{p}\otimes\bar{p}$, we have $\mathcal{E}=\frac{1}{2}[(pp^T+qq^T)\otimes(\bar{p}\bar{p}^T+\bar{q}\bar{q}^T)-(pp^T-qq^T)\otimes(\bar{p}\bar{p}^T-\bar{q}\bar{q}^T)].$
Then
\begin{equation*}
\begin{aligned}
\mathcal{E}(x,x,\bar{x},\bar{x})=&\frac{1}{2}[|(p^Tx)^2+(q^Tx)^2|^2-|(p^Tx)^2-(q^Tx)^2|^2]\\
=&\frac{1}{4}[|p^Tx+q^Tx|^4+|p^Tx-q^Tx|^4-|p^Tx+iq^Tx|^4-|p^Tx-iq^Tx|^4].
\end{aligned}
\end{equation*}
Then the desired CPS decomposition follows.
\end{proof}
\begin{corollary}\label{thm:cpsforreal}
Let $\mathcal{A}\in\mathbb{R}_{cps}^{n^4}\backslash\mathbb{R}_s^{n^4}$, it has a CPS decomposition in the following form.
\begin{equation}\label{equ:realcps}
\mathcal{A}=\sum\limits_{i=1}^{R_1}\lambda_i(a_i^{\otimes2}\otimes\bar{a}_i^{\otimes2}+\bar{a}_i^{\otimes2}\otimes a_i^{\otimes2})+\sum\limits_{j=1}^{R_2}\lambda_j b_j^{\otimes 4},
\end{equation}
where $\lambda_i,\lambda_j\in\mathbb{R}$, $a_i\in\mathbb{C}^n\backslash\mathbb{R}^n$, $b_j\in\mathbb{R}^n$.
\end{corollary}
\begin{proof}
The desired result follows from Theorem \ref{thm:matrix_decomp} and Corollary \ref{thm:construct}.
\end{proof}

Since $a\otimes a\otimes\bar{a}\otimes\bar{a}+\bar{a}\otimes\bar{a}\otimes a\otimes a$ for $a\in\mathbb{C}$ is a real CPS tensor, we pose the following conjecture.
\begin{conjecture}\label{conj:realcps}
Let $\mathcal{A}\in\mathbb{R}_{cps}^{n^4}\backslash\mathbb{R}_s^{n^4}$, it has a CPS rank decomposition in the form of \eqref{equ:realcps}.
\end{conjecture}

In conclusion, for a fourth-order CPS tensor $\mathcal{A}$, a procedure to obtain the CPS decomposition involves the following steps:
\begin{enumerate}
\item Compute the spectral decomposition of $M(\mathcal{A})=\sum_{i = 1}^r \lambda_i e_i e_i^*$. Then folding each $e_i$ back into a symmetric matrix to obtain the orthogonal matrix decomposition $\mathcal{A}=\sum\limits_{i=1}^r\lambda_iE_i\otimes \bar{E_i}$.
\item Compute the SSVD (or spectral decomposition if $E_i$ are real symmetric) for each $E_i$, and obtain the extended "rank-one" decomposition $\mathcal{A}=\sum\limits_{i=1}^R\alpha_i(p_i\otimes p_i\otimes \bar{q}_i\otimes \bar{q}_i+q_i\otimes q_i\otimes \bar{p}_i\otimes \bar{p}_i)$.
\item Form the CPS decomposition as Corollary \ref{thm:construct}.
\end{enumerate}

As an example, we compute the CPS decomposition of $\mathcal{E}^{1122}(1)\in\mathbb{R}_{cps}^{2^4}$ by this procedure. And it is of the same form as \eqref{equ:basis_decomp}, which is also a CPS rank decomposition as asserted in Proposition \ref{pro:basis_decomp}. However, the above procedure cannot yield the minimal CPS decomposition in general, since there may be many redundant items. For the unitary CPS decomposable tensor, that is
\begin{equation}\label{equ:unitary}
\mathcal{A}=\sum\limits_{i=1}^r\lambda_ip_i\otimes p_i\otimes\bar{p}_i\otimes\bar{p}_i,\quad p_i^Hp_j=0~(\forall~i\neq j),
\end{equation}
which is surely a CPS rank decomposition, \cite{fu2019successive} proved that the successive partial-symmetric rank-one algorithm can recover the unitary CPS decomposition.
\begin{remark}
We call $\mathcal{A}\in \mathbb{R}^{n^4}$ skew partial symmetric tensor if
\begin{equation*}
\mathcal{A}_{ijkl}=\mathcal{A}_{jikl}=\mathcal{A}_{ijlk}=-\mathcal{A}_{klij},~1\le i,j,k,l\le n.
\end{equation*}
Then it can be decomposed as
\begin{equation*}
\mathcal{A}=\sum\limits_{i}\lambda_i(U_i\otimes V_i-V_i\otimes U_i),
\end{equation*}
and
\begin{equation*}
\mathcal{A}=\sum\limits_i\lambda_i(p_i\otimes p_i\otimes q_i\otimes q_i-q_i\otimes q_i\otimes p_i\otimes p_i).
\end{equation*}
\begin{proof}
$M(\mathcal{A})$ is skew-symmetric according to the definition of the skew partial-symmetric tensor. Then $M(\mathcal{A})=\sum_{i}\lambda_i(u_iv_i^T-v_iu_i^T)$. The rest of the proof is similar to that for CPS tensors, here we omit it.
\end{proof}
\end{remark}
\subsection{Matrix rank of CPS tensor}
Based on the orthogonal matrix decomposition, we define a new rank for CPS tensors.
\begin{definition}
Let $\mathcal{A}\in \mathbb{C}^{n^4}$ be a CPS tensor, the matrix rank of $\mathcal{A}$, denoted by $rank_m(\mathcal{A})$, is defined as
\[rank_m(\mathcal{A})=min\{r|\mathcal{A}=\sum\limits_{i=1}^r\lambda_iE_i\otimes \bar{E}_i,\lambda_i\in \mathbb{R},E_i\in \mathcal{S}_{\mathbb{C}}^{n}\}.\]
\end{definition}
\begin{corollary}\label{thm:matrix_rank}
Let $\mathcal{A}\in\mathbb{C}^{n^4}_{cps}$, $rank_m(\mathcal{A})\le n(n+1)/2.$
\end{corollary}
\begin{proof}
It follows from the dimension of the symmetric matrices set.
\end{proof}

It is easy to see that $rank_m{\mathcal{A}}=rank(M(\mathcal{A}))$ from the proof of Theorem \ref{thm:matrix_decomp}. $rank_m(\mathcal{A})$ can be seen as a generalization of the strongly symmetric M-rank $rank_{ssm}(\mathcal{A})$ defined by \cite{jiang2018low} for symmetric tensors. Motivated by them, we also define
\[rank_{sm}(\mathcal{A})=min\{r|\mathcal{A}=\sum\limits_{i=1}^r\lambda_i E_i\otimes \bar{E}_i,\lambda_i\in\mathbb{R},E_i\in\mathbb{C}^{n\times n}\}.\]
\begin{theorem}
Let $\mathcal{A}\in \mathbb{C}^{n^4}$ be a CPS tensor, then $rank_m(\mathcal{A})=rank_{sm}(\mathcal{A}).$
\end{theorem}
\begin{proof}
It is obvious that $rank_m(\mathcal{A})\ge rank_{sm}(\mathcal{A}).$ On the other hand, if $rank_{sm}(\mathcal{A})=r$, we have $rank(M(\mathcal{A}))\le r$. Since $rank(M(\mathcal{A}))=rank_m(\mathcal{A})$, we obtain the desired conclusion.
\end{proof}

It is natural to propose the following approximation problem for $\mathcal{A}\in\mathbb{C}_{cps}^{n^4}$, based on previous arguments.
\begin{equation}\label{equ:matrix_approx}
\underset{\substack{\|X\|_F=1,X\in\mathcal{S}_{\mathbb{C}}^n\\ \lambda\in\mathbb{R}}}{\min}\|\mathcal{A}-\lambda X\otimes \bar{X}\|_F.
\end{equation}
This problem can be solved easily by computing the nearest symmetric rank-one problem of $M(\mathcal{A})$, according to the proof of Theorem \ref{thm:matrix_decomp}. Let $A=M(\hat{\mathcal{A}})$, with $\hat{\mathcal{A}}_{ikjl}=\mathcal{A}_{ijkl}$, then $A$ is symmetric, \eqref{equ:matrix_approx} is equivalent to
\begin{equation*}
\underset{\substack{\|X\|_F=1,X\in \mathbb{C}^{n\times n}\\ \lambda\in\mathbb{R}}}{\min}\|A-\lambda X\boxtimes \bar{X}\|_F,
\end{equation*}
which is the so called nearest Kronecker product problem \citep{golub2013matrix}.
\begin{corollary}
Let $\mathcal{A}\in\mathbb{C}_{cps}^{n^4}$ with $\mathcal{A}=\sum_{i=1}^r\lambda_iE_i\otimes \bar{E_i}$, $r=rank_m(\mathcal{A})$, then it has the following properties:
\begin{enumerate}[(1)]
\item If every $E_i$ is rank-one, then $\mathcal{A}$ has a unitarily CPS decomposition as \eqref{equ:unitary}.
\item If $R=\max \{rank(E_i),1\le i\le r\}$, then $r\le rank(\mathcal{A})\le rR^2$.
\item If $r=1$, $\mathcal{A}=\lambda E\otimes\bar{E}$ with $rank(E)=R$, then $rank(\mathcal{A})=R^2$.
\end{enumerate}
\end{corollary}
\begin{proof}
$(1)$ and $(2)$ come from \eqref{equ:proof} directly. If $r=1$ and $rank(E)=R$, then $\mathcal{A}=\lambda\sum_{i,j=1}^R\beta_i\beta_ju_i\otimes u_i\otimes \bar{u}_j\otimes \bar{u}_j$ according to \eqref{equ:proof}. We have $rank(\mathcal{A})\le R^2$. On the other hand, $rank(\mathcal{A})=rank(\mathcal{\tilde{\mathcal{A}}})$, where $\tilde{\mathcal{A}}=\lambda\sum_{i,j=1}^R\beta_i\beta_ju_i\otimes \bar{u}_j\otimes u_i\otimes \bar{u}_j$. Since $\{u_i\boxtimes \bar{u}_j\}_{i,j}$ is an orthogonal base, thus $rank(\mathcal{A})=rank(\tilde{\mathcal{A}})\ge rank(M(\tilde{\mathcal{A}}))=R^2$. $rank(\mathcal{A})=R^2$.
\end{proof}

Inspired by \cite{jiang2018low}, the matrix rank of $\mathcal{A}$ can be used to estimate its bound of CP rank. Then we can approximate the low-CP-rank CPS tensor completion problem by the following  problem:
\begin{equation}\label{equ:lowmrank}
\begin{aligned}
\min ~&rank_{m}(\mathcal{X})\\
{\rm s.t.}~&P_{\Omega}(\mathcal{X})=P_{\Omega}(\mathcal{A}),~\mathcal{X}\in\mathbb{C}_{cps}^{n^4},
\end{aligned}
\end{equation}
where $\mathcal{A},~\mathcal{X}\in\mathbb{C}_{cps}^{n^4}$ and $\Omega$ is the index of sample set. Since $\mathcal{A}$ is CPS, we assume that $\Omega$ enjoys the same symmetry.
\eqref{equ:lowmrank} is relaxed as a low rank approximation problem with nuclear norm regular term,
\begin{equation}\label{equ:completion}
\begin{aligned}
\min~ & \mu\|X\|_*+\frac{1}{2}\|P_{\Omega}(\mathcal{X})-P_{\Omega}(\mathcal{A})\|_F^2\\
{\rm s.t.}~& X=M(\mathcal{X}),~\mathcal{X}\in \mathbb{C}_{cps}^{n^4}.\\
\end{aligned}
\end{equation}
\eqref{equ:completion} can be solved by the Fixed Point Continuation (FPC) algorithm \citep{ma2011fixed}.

Through the above arguments, one of the major benefits of the orthogonal matrix decomposition for fourth-order CPS tensors, is that we can deal with it as a Hermitian matrix. For instance, we can use the successive "rank-one" approximation algorithm to compute the low matrix rank approximation as \cite{fu2019successive}.  The eigenpair corresponding to the orthogonal matrix decomposition was also defined as matrices in literatures \citep{cardoso1990eigen,basser2007spectral}. For the fourth-order CPS tensor, the definition is stated explicitly as follows.
\begin{definition}\label{def:matrixeigen}
Let $\mathcal{A}\in\mathbb{C}_{cps}^{n^4}$, if $\mathcal{A}E=\lambda E$, where $\lambda\in\mathbb{R}$, $0\neq E\in\mathcal{S}_{\mathbb{C}}^n$, then $(\lambda,E)$ is called an eigenpair of $\mathcal{A}$, $\lambda$ and $E$ are corresponding matrix eigenvalue and eigenmatrix, respectively.
\end{definition}

\section{Positive semidefinite CPS tensors}\label{sec:psd}
\cite{nie2020hermitian} studied the positive semidefinite Hermitian tensors and its connection with Hermitian eigenvalues. They also introduced the conjugate and Hermitian sum of squares (HSOS), as well as separable Hermitian tensors, which has close relationship with the nonnegativity of the polynomials. For CPS tensors, \cite{jiang2016characterizing} raised the questions about the nonnegativity of real-valued conjugate complex polynomials.
The cones of nonnegative quartic forms in the term of real fourth-order symmetric tensors has already been discussed extensively \citep{jiang2017cones}. We wonder what the hierarchical relationship of different "positive semidefinite" fourth-order CPS tensors will be, which implies their connection to the nonnegativity of quartic real-valued symmetric conjugate forms.
\subsection{Real case}
We first recall and summarize several positive semidefinite definitions for real CPS tensors.
\cite{qi2004eigenvalues} first proposed the positive semidefinitemess (PSD) for even order symmetric tensor. For a fourth-order symmetric tensor $\mathcal{B}$, as Qi's definition, it is PSD if $\mathcal{B}(x,x,x,x)\ge 0,~\forall x\in\mathbb{R}^n.$ Then the definition was adopted to general even order tensor. Here we concentrate on foruth-order CPS tensors.

\begin{definition}
A fourth-order tensor $\mathcal{A}\in\mathbb{R}_{cps}^{n^4}$ is called positive semidefinite (PSD) if
\begin{equation}\label{equ:real_psd}
\mathcal{A}(x,x,x,x)\ge 0,~\forall x\in\mathbb{R}^n.
\end{equation}
The set of all PSD tensors in $\mathbb{R}_{cps}^{n^4}$ is denoted by $\mathbb{R}CPS^{n^4}_+$.
\end{definition}

\cite{jiang2017cones} introduced the matrix PSD for fourth-order real CPS tensors.
\begin{definition}
A fourth-order tensor $\mathcal{A}\in\mathbb{R}_{cps}^{n^4}$ is called matrix PSD if
\begin{equation*}\label{equ:real_matpsd}
\left<X,\mathcal{A}X\right>\ge 0,~\forall X\in\mathcal{S}_{\mathbb{R}}^n.
\end{equation*}
The set of all matrix PSD tensors in $\mathbb{R}_{cps}^{n^4}$ is denoted by $\vec{S}^{n^2\times n^2}_+$.
\end{definition}
\begin{lemma}
Let $\mathcal{A}\in\mathbb{R}^{n^4}_{cps}$. $\mathcal{A}\in \vec{S}^{n^2\times n^2}_+$ if and only if the matrix eigenvalue in Definition \ref{def:matrixeigen} are nonnegative.
\end{lemma}
According to the above lemma, we reobtain the result of Lemma 3.4 of \cite{jiang2017cones} that $\vec{S}^{n^2\times n^2}_+=\rm{cone}\{E\otimes E|E\in\mathcal{S}_{\mathbb{R}}^n\}$.
\begin{remark}
There exists $\mathcal{A}\in\mathbb{R}_{cps}^{n^4}$, such that $\mathcal{A}(x,x,x,x)$ for $x\in\mathbb{R}^n$ can be written as a sum-of-squares of polynomial functions, that is, $\mathcal{A}(x,x,x,x)=\sum_{i=1}^m(x^TE_ix)^2$ with $E_i\in\mathcal{S}_{\mathbb{R}}^n$, while $\mathcal{A}\notin \vec{S}^{n^2\times n^2}_+$. For instance, consider $\mathcal{A}\in\mathbb{R}_{cps}^{2^4}$, $\mathcal{A}_{1122}=\mathcal{A}_{2211}=1$, then $\mathcal{A}(x,x,x,x)=2x_1^2x_2^2$. On the other hand, $\mathcal{A}$ has a negative matrix eigenvalue in the orthogonal matrix decomposition, thus it will not be matrix PSD. If $x=(i,1)\in\mathbb{C}^2$,  $\mathcal{A}(x,x,\bar{x},\bar{x})=x_1^2\bar{x}_2^2+x_2^2\bar{x}_1^2=-2<0$, it will not always be nonnegative over complex field.
\end{remark}

Yang et al. offered a new matrix PSD, the general PSD, for the fourth-order real CPS tensor, which is closely linked to Yuan's lemma.
\begin{definition}\citep[Definition 2.4]{yang2019extension}
A fourth-order tensor $\mathcal{A}\in\mathbb{R}_{cps}^{n^4}$ is called general PSD if
\begin{equation*}\label{equ:real_generalpsd}
\left<X,\mathcal{A}X\right>\ge 0,~\forall X\in\mathcal{S}_+^n.
\end{equation*}
The set of all general PSD tensors in $\mathbb{R}_{cps}^{n^4}$ is denoted by $PSD_+^{4,n}$.
\end{definition}

It is obvious that $\vec{S}^{n^2\times n^2}_+\subset PSD_+^{4,n}\subset \mathbb{R}CPS_+^{n^4}$. The three sets are all closed convex cones. And according to the bipolar theorem \citep{boyd2004convex} and Carath\'eodory's theorem, we have three primal-dual pairs in $\mathbb{R}_{cps}^{n^4}$.
\begin{lemma}\label{lem:dual}
$\mathbb{R}CPS_{+}^{n^4}=(\rm{cone}\{x^{\otimes 4}|x\in\mathbb{R}^n\})^*$, $(\mathbb{R}CPS_{+}^{n^4})^*=\rm{cone}\{x^{\otimes 4}|x\in\mathbb{R}^n\}$;
$\vec{S}^{n^2\times n^2}_+=(\rm{cone}\{X\otimes X|X\in\mathcal{S}_{\mathbb{R}}^n\})^*=(\vec{S}^{n^2\times n^2}_+)^*$ is self-dual; $PSD_+^{4,n}=(\rm{cone}\{X\otimes X|X\in\mathcal{S}_+^n\})^*$, $(PSD_+^{4,n})^*=\rm{cone}\{X\otimes X|X\in\mathcal{S}_+^n\}$.
\end{lemma}

We will show that a hierarchical relationship
\begin{equation}\label{equ:real_relation}
\vec{S}^{n^2\times n^2}_+\subsetneq PSD_+^{4,n}\subsetneq \mathbb{R}CPS_+^{n^4}
\end{equation}
exists. First for the real symmetric fourth-order tensors, according to Theorem 3.3 of \cite{jiang2017cones}, we have
\begin{equation}\label{equ:jiang}
\rm{cone}\{a^{\otimes 4}|a\in\mathbb{R}^n\}=\rm{sym}~cone\{A\otimes A|A\in\mathcal{S}_+^n\}.
\end{equation}
\begin{theorem}\label{thm:real_sym}
If $\mathcal{A}\in\mathbb{R}_s^{n^4}$, then $\mathcal{A}\in\mathbb{R}CPS_+^{n^4}$ is identity to $\mathcal{A}\in PSD_+^{4,n}$.
\end{theorem}
\begin{proof}
Since $\mathcal{A}$ is real symmetric, $\left<X,\mathcal{A}X\right>=\left<\mathcal{A},X\otimes X\right>=\left<\mathcal{A},\rm{sym}(X\otimes X)\right>$. According to Lemma \ref{lem:dual} and \eqref{equ:jiang}, the desired result then follows.
\end{proof}

However, the following example shows that even for symmetric tensors, $\vec{S}^{n^2\times n^2}_+\subsetneq \mathbb{R}CPS_+^{n^4}$.
\begin{example}\label{exa:notsubset1}
Let $\mathcal{A}\in\mathbb{R}_s^{2^4}$ and $\mathcal{A}=x^{\otimes 4}+y^{\otimes 4}-\frac{1}{8}(x+y)^{\otimes 4}$, where $x=(1,0)^T,~y=(0,1)^T$. Then $\mathcal{A}\in\mathbb{R}CPS_+^{n^4}$, while $\mathcal{A}\notin\vec{S}^{n^2\times n^2}_+$.
\end{example}
\begin{proof}
For any $z=(a,b)^T\in\mathbb{R}^2$, we have
\[\mathcal{A}(z,z,z,z)=a^4+b^4-\frac{1}{8}(a+b)^4\ge 0.\]
Let $X=(\begin{smallmatrix}a&c\\c&b\end{smallmatrix})\in \mathcal{S}_{\mathbb{R}}^n$, we have
\[\left<X,\mathcal{A}X\right>=a^2+b^2-\frac{1}{8}(a+b+2c)^2,\]
which might be negative. For example, when taking $a=b=0,~c=1$, $\left<X,\mathcal{A}X\right>=-\frac{1}{2}<0$.
\end{proof}

Finally, we present two examples to complete the relationship \eqref{equ:real_relation} for general fourth-order real CPS tensors.
\begin{example}
Let $\mathcal{A}\in\mathbb{R}_{cps}^{2^4}$, with $\mathcal{A}_{1111}=\mathcal{A}_{2222}=1,~\mathcal{A}_{1212}=
\mathcal{A}_{2112}=\mathcal{A}_{1221}=\mathcal{A}_{2121}=1.5,
~\mathcal{A}_{1122}=\mathcal{A}_{2211}=-2$, and all other entries are zero. Thus there exists $\mathcal{A}\in \mathbb{R}CPS_+^{2^4}\backslash PSD_+^{4,2}$.
\end{example}
\begin{proof}
For any nonzero vector $x=(x_1,x_2)^T\in \mathbb{R}^2$, we have
$$\mathcal{A}x^4=(x_1^2+x_2^2)^2>0.$$
For $X=\left(\begin{smallmatrix}5&1\\-1&2\end{smallmatrix}\right)\in\mathcal{S}_+^2$,
$$\left<X,\mathcal{A}X\right>=-5<0.$$
Thus, the required result is obtained.
\end{proof}
\begin{example}\label{exa:notsubset2}
Let $\mathcal{A}\in\mathbb{R}_{cps}^{2^4}$, $\mathcal{A}=\lambda E_1\otimes E_1-E_2\otimes E_2$, where $E_1=\left(\begin{smallmatrix}-2&1\\1&-3\end{smallmatrix}\right)$ and $E_2=\left(\begin{smallmatrix}2&5\\5&2\end{smallmatrix}\right)$. There exists proper $\lambda>0$, such that $\mathcal{A}\in PSD_+^{4,2}\backslash\vec{S}^{2^2\times 2^2}_+$.
\end{example}
\begin{proof}
For any $X=\left(\begin{smallmatrix}a&c\\c&b\end{smallmatrix}\right)\in\mathcal{S}_+^n$, we have $a\ge0,~b\ge0,~ab-c^2\ge0$. To find a proper $\lambda>0$ such that $\mathcal{A}\in PSD_+^{4,2}$, we only need to find a $\lambda$ such that
\begin{equation*}
\begin{aligned}
\left<X,\mathcal{A}X\right>&=\lambda(2a+3b-2c)^2-(2a+2b+10c)^2\nonumber\\
(a,b\ge 0,ab\ge c^2)&\ge\lambda(2a+3b-2\sqrt{ab})^2-(2a+2b+10\sqrt{ab})^2\label{equ:compu1}\\
&\ge 0.
\end{aligned}
\end{equation*}
That is, $\lambda\ge(\frac{2a+3b-2\sqrt{ab}}{2a+2b+10\sqrt{ab}})^2,~\forall a\ge0,~ b\ge0$. Let $f(t)=\frac{2+2t^2+10t}{2+3t^2-2t}~(t=\frac{b}{a},~t\in [0,\infty])$, then $f(t)\ge 0$ is bounded by a constant $C$. We take $\lambda=C$, then $\mathcal{A}$ is general PSD.

On the other hand, $\left<E_1,E_2\right>=0$, then $\left<E_2,\mathcal{A}E_2\right>=-\|E_2\|_F^2<0$. Thus $\mathcal{A}$ is not matrix PSD.
\end{proof}
\subsection{Complex case}
For the complex case, we may give the following definitions analogous to real case.
\begin{definition}
A fourth-order tensor $\mathcal{A}\in\mathbb{C}_{cps}^{n^4}$ is called positive semidefinite (PSD) if
\begin{equation*}\label{equ:complex_psd}
\mathcal{A}(x,x,\bar{x},\bar{x})\ge 0,~\forall x\in\mathbb{C}^n.
\end{equation*}
The set of all PSD tensors in $\mathbb{C}_{cps}^{n^4}$ is denoted by $CPS^{n^4}_+$.
\end{definition}
\begin{definition}
A fourth-order tensor $\mathcal{A}\in\mathbb{C}_{cps}^{n^4}$ is called matrix PSD if
\begin{equation*}\label{equ:complex_matpsd}
\left<X,\mathcal{A}X\right>\ge 0,~\forall X\in\mathcal{S}_{\mathbb{C}}^n.
\end{equation*}
The set of all matrix PSD tensors in $\mathbb{C}_{cps}^{n^4}$ is denoted by $\mathbb{C}\vec{S}^{n^2\times n^2}_+$.
\end{definition}

Similar to the real case, $\mathcal{A}\in\mathbb{C}\vec{\mathcal{S}}_+^{n^2\times n^2}$ if and only if $\mathcal{A}\in\rm{cone}\{E\otimes \bar{E}|E\in\mathcal{S}_{\mathbb{C}}^n\}$. Since there is a bijection between CPS tensors and corresponding quartic real-valued symmetric conjugate forms \citep{jiang2016characterizing}, $\mathcal{A}\in\mathbb{C}\vec{\mathcal{S}}_+^{n^2\times n^2}$ is also equivalent to that $\mathcal{A}(x,x,\bar{x},\bar{x})=\sum_{i=1}^r|(x^TE_ix)|^2$. Such a complex quartic form $\mathcal{A}(x,x,\bar{x},\bar{x})$ belongs to the so called HSOS defined by \cite{nie2020hermitian}. It is obvious that $\mathcal{A}\in CPS_+^{n^4}$ if $\mathcal{A}\in\mathbb{C}\vec{\mathcal{S}}_+^{n^2\times n^2}$.
\begin{definition}
A fourth-order tensor $\mathcal{A}\in\mathbb{C}_{cps}^{n^4}$ is called general PSD if
\begin{equation*}\label{equ:complex_generalpsd}
\left<X,\mathcal{A}X\right>\ge 0,~\forall X\in\mathcal{H}_+^n.
\end{equation*}
The set of all general PSD tensors in $\mathbb{C}_{cps}^{n^4}$ is denoted by $HPSD_+^{4,n}$.
\end{definition}
\begin{lemma}\label{lem:skew-sym}
$\mathbb{C}\vec{S}^{n^2\times n^2}_+\subset HPSD_+^{4,n}$.
\end{lemma}
\begin{proof}
Let $\mathcal{A}\in\mathbb{C}\vec{S}^{n^2\times n^2}_+$. For any $X\in\mathcal{H}_+^n$, we have $X=M+iN$ where $M\in\mathcal{S}_+^n$ and $N$ is a real skew-symmetric matrix. For any $Y\in \mathcal{S}^n_{\mathbb{R}}$, we have
\begin{equation*}
\begin{aligned}
\left<N,\mathcal{A}Y\right>=\sum\limits_{ijkl}\mathcal{A}_{ijkl}Y_{ij}N_{kl}
=\sum\limits_{ijkl}\mathcal{A}_{ijkl}Y_{ij}(-N_{lk})=-\sum\limits_{ijkl}\mathcal{A}_{ijlk}Y_{ij}N_{lk}
=-\left<N,\mathcal{A}Y\right>.
\end{aligned}
\end{equation*}
Thus, $\left<N,\mathcal{A}Y\right>=0$. The third equality comes from the partial symmetry of $\mathcal{A}$. And $\left<Y,\mathcal{A}N\right>=0$ also holds for any $Y\in\mathcal{S}^n_{\mathbb{R}}$.
Then we obtain
\begin{equation}\label{equ:mpsd_equal}
\begin{aligned}
\left<X,\mathcal{A}X\right>&=\left<M,\mathcal{A}M\right>+\left<N,\mathcal{A}N\right>
+i\left<M,\mathcal{A}N\right>-i\left<N,\mathcal{A}M\right>\\
&=\left<M,\mathcal{A}M\right>\ge 0.
\end{aligned}
\end{equation}
That is, $\mathcal{A}\in HPSD_+^{4,n}$.
\end{proof}
\begin{corollary}\label{thm:mpsd_equal}
Let $\mathcal{A}\in\mathbb{C}_{cps}^{n^4}$, $\mathcal{A}\in HPSD_+^{4,n}$ if and only if $\mathcal{A}\in PSD_+^{4,n}$.
\end{corollary}
\begin{proof}
The necessary condition is obvious, the sufficient condition follows from \eqref{equ:mpsd_equal}.
\end{proof}
\begin{theorem}\label{thm:psd_equal}
Let $\mathcal{A}\in\mathbb{R}_{cps}^{n^4}$, $\mathcal{A}\in \mathbb{C}\vec{S}^{n^2\times n^2}_+$ if and only if $\mathcal{A}\in \vec{S}^{n^2\times n^2}_+$.
\end{theorem}
The proof is similar to Lemma \ref{lem:skew-sym}, we omit it. With a modification of Examples \ref{exa:notsubset1} and \ref{exa:notsubset2}, we obtain that $\mathbb{C}\vec{S}_+^{n^2\times n^2}\subsetneq CPS_+^{n^4}$ and $\mathbb{C}\vec{S}_+^{n^2\times n^2}\subsetneq HPSD_+^{4,n}$. The main differences from the real case is encountered in the relationship between $CPS_+^{n^4}$ and $HPSD_+^{4,n}$.
\begin{theorem}
If $\mathcal{A}\in\mathbb{C}_{cps}^{n^4}$ is symmetric, then $A\in CPS_+^{n^4}$ is identity to $\mathcal{A}\in HPSD_+^{4,n}$.
\end{theorem}
\begin{proof}
If $\mathcal{A}\in HPSD_+^{4,n}$, according to the symmetry of $\mathcal{A}$, we have
\[\mathcal{A}(x,x,\bar{x},\bar{x})=\mathcal{A}(x,\bar{x},\bar{x},x)
=\left<xx^H,\mathcal{A}xx^H\right>\ge 0.\]
Thus, $\mathcal{A}\in CPS_+^{n^4}$.

On the other hand, if the CPS tensor $\mathcal{A}$ is also symmetric, then $\mathcal{A}$ must be a real symmetric tensor. If $\mathcal{A}\in CPS_+^{n^4}$, then $\mathcal{A}\in\mathbb{R}CPS_+^{n^4}$. According to Theorem \ref{thm:real_sym}, $\mathcal{A}\in PSD_+^{4,n}$. Combined with Corollary \ref{thm:mpsd_equal}, we have $\mathcal{A}\in HPSD_+^{4,n}$.
\end{proof}

Finally, we present two examples that show $HPSD_+^{4,n}\not\subset CPS_+^{n^4}$ and $CPS_+^{n^4}\not\subset HPSD_+^{4,n}$.
\begin{example}
Let $\mathcal{A}\in \mathbb{R}_{cps}^{2^4}$, with $\mathcal{A}_{1111}=\mathcal{A}_{2222}=1$, $\mathcal{A}_{1212}=\mathcal{A}_{2112}=\mathcal{A}_{1221}=\mathcal{A}_{2121}=1.5$, $\mathcal{A}_{1122}=-2-2i,~\mathcal{A}_{2211}=-2+2i$, and all other entries are zero. Then $\mathcal{A}\in CPS^{2^4}_+$ and $\mathcal{A}\notin HPSD^{4,2}_+$.
\end{example}
\begin{proof}
For any nonzero vector $x=(x_1,x_2)^T\in \mathbb{R}^2$, we have
\begin{equation*}
\begin{aligned}
\mathcal{A}(x,x,\bar{x},\bar{x})&=|x_1|^4+|x_2|^4+6|x_1|^2|x_2|^2+4Re(ix_1x_1\bar{x}_2\bar{x}_2)-4Re(x_1x_1\bar{x}_2\bar{x}_2)\\
&\ge |x_1|^4+|x_2|^4+6|x_1|^2|x_2|^2-8|x_1|^2|x_2|^2=(|x_1|^2-|x_2|^2)^2\ge 0.
\end{aligned}
\end{equation*}
For $X=\left(\begin{smallmatrix}5&i\\-i&2\end{smallmatrix}\right)\in \mathcal{H}_+^2$,
\begin{equation*}
\begin{aligned}
\left<X,\mathcal{A}X\right>&=|X_{11}|^2+|X_{22}|^2+1.5(|X_{12}|^2+|X_{11}|^2+2Re(\bar{X_{12}}X_{21}))\\
&-2i\bar{X}_{11}X_{22}+2i\bar{X}_{22}X_{11}-2\bar{X}_{11}X_{22}-2\bar{X}_{22}X_{11}\\
&=-11<0.
\end{aligned}
\end{equation*}
Thus, the required result is obtained.
\end{proof}
\begin{example}
Let $\mathcal{A}$ be defined as Example \ref{exa:notsubset2}, it is easy to verify that $\mathcal{A}\in HPDS^{4,n}_+$ according to Corollary \ref{thm:mpsd_equal}; while for $x=(i,\frac{\sqrt{5}}{3}+i\frac{1}{3})^T$, $\mathcal{A}(x,x,\bar{x},\bar{x})<0$, so $\mathcal{A}\notin CPS_+^{n^4}$.
\end{example}
\begin{theorem}
Let $\mathcal{A}\in CPS_+^{n^4}$, if $r=rank_{cps}(\mathcal{A})\le n$, then $\mathcal{A}\in \mathbb{C}\vec{S}_+^{n^2\times n^2}$ and $\mathcal{A}\in HPSD^{4,n}_+$ generally.
\end{theorem}
\begin{proof}
According to the assumption, there exists a CPS decomposition, $\mathcal{A}=\sum_{i=1}^r\lambda_ia_i^{\otimes^2}\otimes\bar{a}_i^{\otimes^2}$, where $\lambda_i\in\mathbb{R}$, $a_i\in\mathbb{C}^n$, for $i=1,2,\cdots,r$. Since $r\le n$, according to \citep[Lemma 5.2]{comon2008symmetric}, the vectors of the set $\{a_1,\cdots,a_r\}$ are generally independent. Combined with $\mathcal{A}\in CPS_+^{n^4}$, $\lambda_i\ge 0$ for $i=1,2,c\dots,r$. Thus $\langle X,\mathcal{A}X\rangle\ge 0$, for any $X\in\mathcal{S}_{\mathbb{C}}^n$. $\mathcal{A}\in\mathbb{C}\vec{S}_+^{n^2\times n^2}$ and $\mathcal{A}\in HPSD^{4,n}_+$.
\end{proof}
\section{Application examples}\label{sec:app}
At last, let us see several CPS tensors as examples for their decomposition, CP rank bound or positive semidefinitess, based on previous arguments. In the bellow, $\rm{tridiag}(a,b,c)$ represents a tridiagonal matrix with $a,~b,~c$ as the value in sub-diagonals in order. $\lambda_{min}(C)$ stands for the smallest eigenvalue of matrix $C$.
\subsection{CPS tensor from BEC}
Bose-Einstein condensate(BEC) is the phenomenon
that a large number of particles in boson system occupy the
quantum ground state at a temperature below a certain critical
value, and aggregate into a phenomenon of "large particles". The ground state of non-rotating Bose-Einstein condensate (BEC) is defined as the minimizer of the following energy functional minimization problem:
\begin{equation}\label{equ:energy}
\left\{
\begin{array}{lrc}
\min \quad E(\phi(\textbf{x})):= \int_{R^d}[\frac{1}{2}|\nabla\phi(\textbf{x})|^2+V(\textbf{x})|\phi(\textbf{x})|^2+\frac{\beta}{2}|\phi(\textbf{x})|^4]d\textbf{x} \\
{\rm s.t.}\quad \int_{\mathbb{R}^d}|\phi(\textbf{x})|^2d\textbf{x}=1, ~E(\phi)<\infty.
\end{array}
\right.
\end{equation}
where $\textbf{x}\in\mathbb{R}^d$ is the spatial coordinate vector, $V(\textbf{x})$ is an external trapping potential, and the given constant $\beta$ is the dimensionless interaction coefficient \citep{bao2013mathematical}. After a suitable discretization such as the finite difference, sine pseudospectral and fourier pseudospectral methods, \eqref{equ:energy} becomes a nonconvex quartic minimization problem with a single spherical constraint:
\begin{equation}\label{equ:different}
\left\{
\begin{array}{lrc}
\underset{x\in \mathbb R^n}{\min} \quad \frac{\alpha}{2}\sum\limits_{i=1}^{n} x_i^4+x^{T}Bx\\
{\rm s.t.} \quad \|x\|_2^2=1,
\end{array}
\right.
\end{equation}
where $B$ is a symmetric matrix, representing the sum of the discretized Laplacian operator and a positive diagonal matrix. Due to the spherical constraint, \eqref{equ:different} can be converted into the following equivalent form:
\begin{equation}\label{equ:CPS_BEC}
\left\{
\begin{array}{lrc}
\underset{x\in \mathbb R^n}{\min} \quad \frac{\alpha}{2}\langle\mathcal{I},x^{\otimes 4}\rangle+\frac{1}{2}\langle B\otimes I+I\otimes B,x^{\otimes 4}\rangle\\
{\rm s.t.} \quad \|x\|_2^2=1,
\end{array}
\right.
\end{equation}
where $\mathcal{I}\in \mathbb{R}^{n^4}$ is the diagonal tensor with diagonal entries equal to one and $I$ is the identity matrix. It is obvious that \eqref{equ:CPS_BEC} can be seen as the best rank-1 tensor approximation to a fourth-order real CPS tensor
\[\mathcal{F}=\alpha\mathcal{I}+\mathcal{A},\]
where
\[\mathcal{A}=B\otimes I+I\otimes B.\]
\begin{example}\label{exa:app1}
We take the space domain $D=[0,1]$ for $d=1$ with Dirichlet boundary conditions, $\beta>0$ and $V(x)=x^2/2$. Suppose \eqref{equ:different} is obtained by the finite difference method with $n=4$. Then $$B=(n+1)^2\cdot[tridiag(-1,2,-1)+diag([1^2,2^2,\cdots,n^2])],$$ and $\mathcal{F},~\mathcal{A}\in\mathbb{R}^{4^4}_{cps}$. We observe that $\mathcal{A}\in \mathbb{R}CPS_+^{n^4}\backslash\vec{S}^{n^2\times n^2}_+$, $rank(\mathcal{A})=16$ and $rank(\mathcal{F})\le 20$.
\end{example}
\begin{proof}
It is obvious that $\mathcal{A}\in \mathbb{R}CPS_+^{n^4}$ and $rank(M(\mathcal{A}))=2$. We can compute the orthogonal matrix decomposition of $\mathcal{A}$
\begin{equation*}
\mathcal{A}=218.7885E_1\otimes E_1-16.3885E_2\otimes E_2,
\end{equation*}
which implies that $\mathcal{A}\notin \vec{S}^{n^2\times n^2}_+$. The spectral decompositions for $E_1$ and $E_2$ are
\begin{equation*}
\begin{aligned}
E_1 = 0.6579p_1p_1^T+0.5486p_2p_2^T+0.3110p_3p_3^T+0.4115p_4p_4^T;\\
E_2 = 0.5099p_1p_1^T+0.1105p_2p_2^T-0.7579p_3p_3^T-0.3950p_4p_4^T.
\end{aligned}
\end{equation*}
Here $rank(E_1)=rank(E_2)=4$ and they contain the same eigenvectors coincidentally. Thus, $\mathcal{A}$ has a CP decomposition in the form of
\begin{equation}\label{equ:cp_decom}
\mathcal{A}=\sum\limits_{i=1}^4\sum\limits_{j=1}^4\lambda_{ij}p_i^{\otimes 2}\otimes p_j^{\otimes 2},~\lambda_{ij}\neq 0.
\end{equation}
We have $rank(\mathcal{A})\le16$. On the other hand, let $\hat{\mathcal{A}}$ be given as $\hat{\mathcal{A}}_{ijkl}=\mathcal{A}_{ikjl}$, then
\[M(\hat{\mathcal{A}})=\sum\limits_{i=1}^4\sum\limits_{j=1}^4\lambda_{ij}
(p_i\boxtimes p_j)(p_i\boxtimes p_j)^T.\]
From the orthogonality of $\{p_i\boxtimes p_j\}_{i,j}$, $rank(\mathcal{A})\ge rank(M(\hat{\mathcal{A}}))=16$. Thus, \\$rank(\mathcal{A})=16$ and \eqref{equ:cp_decom} is a CP rank decomposition for $\mathcal{A}$.

Since $\mathcal{I}$ has a CP rank decomposition as $\mathcal{I}=\sum_{i=1}^4e_i^{\otimes 4}$ where $e_i$ is the $ith$ identity vector, for $i=1,2,3,4$, we have $rank(\mathcal{F})\le 20$.
\end{proof}
\subsection{The general PSD cauchy tensor}
\begin{example}
Let $\mathcal{C}\in \mathbb{C}_{cps}^{n^4}$ be a Cauchy tensor \citep{chen2015positive}, defined by
\[\mathcal{C}_{ijkl}=\frac{1}{c_i+c_j+c_k+c_l},~i,j,k,l=1,2\cdots,n.\]
$c=(c_1,c_2,\cdots,c_n)^T\in\mathbb{R}^n$ is the generating vector of $\mathcal{C}$. Then $\mathcal{C}\in HPSD_+^{4,n}$ if and only if $c> 0$.
\end{example}
\begin{proof}
According to \citep[Theorem 2.1]{chen2015positive}, $\mathcal{C}\in \mathbb{R}CPS_+^{n^4}$ if and only if $c> 0$. Since $\mathcal{C}$ is real symmetric, $\mathcal{C}\in \mathbb{R}CPS_+^{n^4}$ is identity to $\mathcal{C}\in PSD^{4,n}_+$ according to Theorem \ref{thm:real_sym}. Combined with Corollary \ref{thm:mpsd_equal}, we obtain that $c> 0$ is equivalent to $\mathcal{C}\in HPSD_+^{4,n}$.
\end{proof}
\subsection{CPS tensor from quadratic eigenvalue problem}
The quadratic eigenvalue problem (QEP) \citep{tisseur2001quadratic} is to find scalars $\lambda$ and nonzero $x,y$ satisfying
\[(\lambda^2M+\lambda C+K)x=0,y^H(\lambda^2M+\lambda C+K)=0,\]
where $M,C,K$ are $n$ by $n$ complex matrices, $\lambda$ is an eigenvalue and $x$, $y$ are the corresponding right and left eigenvectors, respectively. The QEP has extensive applications in areas such as the dynamic analysis of mechanical systems in acoustics and linear stablility of flows in fluid mechanics.

One important class of QEPs is from overdamped systems. In structural mechanics, the system is said to be overdamped when the following overdamping condition is satisfied.
\begin{equation}\label{equ:overdamp}
\underset{\|x\|_2=1}{\min}\quad[(x^HCx)^2-4(x^HMx)(x^HKx)]>0,
\end{equation}
where $M,C$ are symmetric positive definite and $K$ is symmetric positive semidefinite. If a system is overdamped, the corresponding QEP will have many special properties including that all the eigenvalues of QEP are real and nonpositive and there is a gap between $n$ largest and $n$ smallest eigenvalues.

Clearly, \eqref{equ:overdamp} is equivalent to
\begin{equation*}
\underset{\|x\|_2=1}{\min}\left<\mathcal{A},\bar{x}\otimes x\otimes\bar{x}\otimes x\right>>0,
\end{equation*}
where $\mathcal{A}=C\otimes C-2(M\otimes K+K\otimes M)$ is a real CPS tensor.
\begin{example}\label{exa:quadraticeigen}
When $M=I$, $C=\tau~ \rm{tridiag}(-1,3,-1)$, $K=\mu~\rm{tridiag}(-1,3,-1)$, it models a connected damped mass-spring system  \citep{tisseur2001quadratic}. We choose $n=50$, $\tau=10$ and $\mu = 6$. In this case, $rank(\mathcal{A})=2500$, $\mathcal{A}\notin \vec{S}^{n^2\times n^2}_+$ and $\mathcal{A}\in \mathbb{R}CPS_+^{n^4}$. The system is overdamped.
\end{example}
\begin{proof}
It is obvious that $rank(M(\mathcal{A}))\le 2$. After computation, the orthogonal matrix decomposition of $\mathcal{A}$ is
\[\mathcal{A}=\lambda_1E_1\otimes E_1+\lambda_2E_2\otimes E_2,\]
where $\lambda_1=-13.7775$, $\lambda_2=51214$ and $E_1,E_2\in\mathcal{S}_{\mathbb{R}}^n$. So it is obvious that $\mathcal{A}\notin \vec{S}^{n^2\times n^2}_+$. Again through the spectral decomposition for $E_1$ and $E_2$, we find that $E_2\in \mathcal{S}_+^n$. Then for any $x\in \mathbb{C}^n$ with $\|x\|_2=1$,
\[\left<\mathcal{A},\bar{x}\otimes x\otimes\bar{x}\otimes x\right>\ge \lambda_2\lambda_{min}^2(E_2)+\lambda_1\|E_1\|_2^2=76.6691>0.\]
Thus the system is overdamped. It is obvious that $\mathcal{A}\in\mathbb{R}CPS_+^{n^4}$.

For the CP rank of $\mathcal{A}$, since $E_1E_2=E_2E_1$ and $rank(E_1)=rank(E_2)=50$, $E_1$ and $E_2$ contain the same eigenvectors. The rest of the proof is the same as the proof for Example \ref{exa:app1}.
\end{proof}
\begin{remark}
For Example \ref{exa:quadraticeigen}, if we directly estimate $\left<\mathcal{A},\bar{x}\otimes x\otimes\bar{x}\otimes x\right>$ by
\[\left<\mathcal{A},\bar{x}\otimes x\otimes\bar{x}\otimes x\right>\ge \lambda_{min}^2(C)-4\|M\|_2\|K\|_2=-19.1489<0,~\forall \|x\|_2=1,\]
we cannot obtain the system is overdamped. However, through the orthogonal matrix decomposition of $\mathcal{A}$ as the above proof, we obtain a much better estimation of the minimal value.
\end{remark}
\section{Conclusions and future work}\label{sec:conclude}
This paper studies fourth-order CPS tensors, CPS decomposition and orthogonal matrix decomposition for them, and different positive semidefiniteness based on different decompositions as well as their relationship. We note that for real CPS tensors, the decompositions might have some different properties from the complex case based on the property of real Hermitian decomposition. Orthogonal matrix decomposition is emphasized since they can be treated as the decomposition of the square unfolding matrix and preserving nice properties. Based on it, we offer a procedure to compute a CPS decomposition and therefore the CPS decomposability of CPS tensors is reobtained. It also gives a bound for the CP rank. Finally, we give the relationship of different positive semidefinite fourth-order CPS tensor cones, which were proposed in different literatures, in real and complex case. It also implies the relationship of the nonnegativity of quartic real-valued symmetric conjugate forms with SOS property.

Several questions need to be studied in the future. A basic question is how to determine CPS rank and compute the CPS rank decomposition for the other basis CPS tensors, just as the nice Hermitian rank decomposition for basis Hermitian tensors obtained by \cite{nie2020hermitian}. Decompositions for real CPS tensors that are not symmetric are of particular interest, which we only discuss intially and left the Conjecture \ref{conj:realcps}. We have discussed the relationship of different positive semideifnite CPS tensor cones. However, whether a CPS fourth-order tensor is general PSD or not, is still hard to judge. Another topic of interest is to study their relationship with the convex quartic conjugate complex forms.


%
%

\bibliographystyle{spbasic}      
\bibliography{ref}   

\end{document}